\documentclass[11pt]{amsart}
\usepackage[dvipsnames]{xcolor}
\usepackage{amssymb,amsmath,amsthm,enumerate,mathtools,mathptmx}
\usepackage{stmaryrd}
\usepackage[new]{old-arrows}
\usepackage{tikz-cd} 
\usepackage{tikz}
\usepackage{multirow}
\usepackage[utf8]{inputenc}
\usepackage{bm}			

\usepackage{hyperref}
\hypersetup{%
	colorlinks = true,
	linkcolor = BrickRed,
	citecolor = Green,
	urlcolor   = blue,
  filecolor = red,
}

\usepackage{caption} 
\usepackage{cleveref}
\usepackage{enumitem}
\usepackage[margin=0.9in]{geometry}
\usepackage{parskip}

\usepackage[backend=biber,style=alphabetic,doi=false,isbn=false,url=false,eprint=false,maxbibnames=5,minbibnames=5,mincitenames=5,maxcitenames=5,maxalphanames=5,minalphanames=5,backref=true]{biblatex}
\DeclareFieldFormat{extraalpha}{#1}
\DeclareLabelalphaTemplate{%
  \labelelement{%
    \field[final]{shorthand}
    \field{label}
    \field[strwidth=2,strside=left,ifnames=1]{labelname}
    \field[strwidth=1,strside=left]{labelname}
  }
}
\DefineBibliographyStrings{english}{%
  backrefpage={},
  backrefpages={}
}

\usepackage{xpatch}
\DeclareFieldFormat{backrefparens}{\addperiod#1}
\xpatchbibmacro{pageref}{parens}{backrefparens}{}{}

\newtheorem{thm}{Theorem}[section]
\newtheorem{lem}[thm]{Lemma}
\newtheorem{prop}[thm]{Proposition}
\newtheorem{cor}[thm]{Corollary}
\newtheorem{por}[thm]{Porism}

\theoremstyle{definition}
\newtheorem{rem}[thm]{Remark}
\newtheorem{defn}[thm]{Definition}

\numberwithin{equation}{section}

\crefname{thm}{theorem}{theorems}
\crefname{rem}{remark}{remarks}
\crefname{prop}{proposition}{propositions}
\crefname{lem}{lemma}{lemmas}
\crefname{por}{porism}{porisms}
\crefname{identity}{identity}{identities}
\crefname{equation}{}{}

\DeclareMathOperator{\GL}{GL}
\DeclareMathOperator{\PGL}{PGL}
\DeclareMathOperator{\SL}{SL}
\DeclareMathOperator{\PSL}{PSL}
\DeclareMathOperator{\OO}{O}
\DeclareMathOperator{\SO}{SO}
\DeclareMathOperator{\Sp}{Sp}
\DeclareMathOperator{\MM}{M}
\DeclareMathOperator{\diag}{diag}
\DeclareMathOperator{\Sym}{Sym}
\DeclareMathOperator{\chr}{char}
\DeclareMathOperator{\Hilb}{Hilb}
\DeclareMathOperator{\Jac}{Jac}
\DeclareMathOperator{\id}{id}

\DeclareMathOperator{\rank}{rank}
\DeclareMathOperator{\codim}{codim}
\DeclareMathOperator{\Cl}{Cl}
\DeclareMathOperator{\Hom}{Hom}
\DeclareMathOperator{\Frac}{Frac}

\newcommand{\md}[1]{{\left\lvert #1 \right\lvert}}
\newcommand{\mdd}[1]{{\lvert #1 \lvert}}
\newcommand{\deff}[1]{{\color{Turquoise}#1}}
\newcommand{\andd}{\quad\text{and}\quad}
\newcommand{\into}{\longhookrightarrow}
\DeclareRobustCommand{\onto}{\relbar\joinrel\twoheadrightarrow}
\newcommand{\tr}{\operatorname{tr}}
\newcommand{\trace}{\operatorname{trace}}
\newcommand{\sym}{W}
\newcommand{\smatrix}[1]{\left(\begin{smallmatrix} #1 \end{smallmatrix}\right)}
\renewcommand{\sl}{\mathfrak{sl}_{2}}
\newcommand{\gl}{\mathfrak{gl}_{2}}
\newcommand{\oo}{\mathfrak{o}_{2}}

\makeatletter
\preto\env@matrix{\renewcommand{\arraystretch}{1}}
\makeatother

\newcommand{\generic}{\left[\begin{smallmatrix} a & b \\ c & d \end{smallmatrix}\right]}
\newcommand{\traceless}{\left[\begin{smallmatrix} a & b \\ c & -a \end{smallmatrix}\right]}
\newcommand{\symmetric}{\left[\begin{smallmatrix} a & b \\ b & d \end{smallmatrix}\right]}
\newcommand{\bgeneric}{\begin{bmatrix} a & b \\ c & d \end{bmatrix}}
\newcommand{\btraceless}{\begin{bmatrix} a & b \\ c & -a \end{bmatrix}}
\newcommand{\bsymmetric}{\begin{bmatrix} a & b \\ b & d \end{bmatrix}}
\newcommand{\balternating}{\begin{bmatrix} 0 & b \\ -b & 0 \end{bmatrix}}

\newcommand{\mysetminusD}{\hbox{\tikz{\draw[line width=0.6pt,line cap=round] (3pt,0) -- (0,6pt);}}}
\newcommand{\mysetminusT}{\mysetminusD}
\newcommand{\mysetminusS}{\hbox{\tikz{\draw[line width=0.45pt,line cap=round] (2pt,0) -- (0,4pt);}}}
\newcommand{\mysetminusSS}{\hbox{\tikz{\draw[line width=0.4pt,line cap=round] (1.5pt,0) -- (0,3pt);}}}
\newcommand{\mysetminus}{\mathbin{\mathchoice{\mysetminusD}{\mysetminusT}{\mysetminusS}{\mysetminusSS}}}

\let\subset\subseteq

\let\ge\geqslant
\let\le\leqslant
\let\mapsto\longmapsto
\let\to\longrightarrow

\begin{document}

\title[Polynomial invariants of conjugation over finite fields]{Polynomial invariants of classical subgroups of \texorpdfstring{$\GL_{2}$}{GL(2)}: \\ Conjugation over finite fields}

\author{Aryaman Maithani}
\address{Department of Mathematics, University of Utah, 155 South 1400 East, Salt Lake City, UT~84112, USA}
\email{maithani@math.utah.edu}

\thanks{The author was supported by NSF grants DMS~2101671 and DMS~2349623.}



\begin{abstract} 
	Consider the conjugation action of the general linear group $\GL_{2}(K)$ on the polynomial ring $K[X_{2 \times 2}]$. 
	When $K$ is an infinite field, the ring of invariants is a polynomial ring generated by the trace and the determinant. 
	We describe the ring of invariants when $K$ is a finite field, and show that it is a hypersurface. 
	We also consider the other classical subgroups, and the polynomial rings corresponding to other subspaces of matrices such as the traceless and symmetric matrices. 
	In each case, we show that the invariant ring is either a polynomial ring or a hypersurface. 
\end{abstract}

\maketitle

\section{Introduction} \label{sec:introduction}

	Let $K$ be a field, and consider the \emph{conjugation action} of the general linear group $G = \GL_{2}(K)$ on the polynomial ring 
	$S = K[X_{2 \times 2}] = K[x_{11}, x_{12}, x_{21}, x_{22}]$: 
	if $X$ denotes the square matrix of variables, then the element $\sigma \in G$ acts by mapping $x_{ij}$ to the $(i, j)$-th entry of $\sigma^{-1} X \sigma$.
	When $K$ is infinite, the ring of invariants is generated by the trace and determinant, i.e., $S^{G} = K[\trace X, \det X]$. 
	More generally, one may consider the conjugation action of $\GL_{n}(K)$ on $K[X_{n \times n}]$. 
	When the field $K$ is infinite, $K[X_{n \times n}]^{\GL_{n}(K)}$ is classically known to be generated by the coefficients of the characteristic polynomial of $X$, see for example \Cite[Example 2.1.3]{DerksenKemper}.
	It is easy to see that this can not hold when the field $K$ is finite: 
	indeed, $G$ is then a finite group, and thus the Krull dimensions	of $S$ and $S^{G}$ must be the same, telling us that the invariant subring is larger. 
	The first part of the paper concerns the following description of $S^{G}$ when $K$ is finite, expanding on the work of Larry Smith~\Cite{Smith:Conjugation}. 

	\begin{thm} \label{mainthm}
		Let $K$ be a finite field with $q$ elements. 
		Consider the conjugation action of the general linear group $G \coloneqq \GL_{2}(K)$ on the polynomial ring $S \coloneqq K[X_{2 \times 2}]$.
		Let 
		$f_{1} = \trace$, 
		$f_{2} = \det$, 
		$f_{3} = \mathcal{P}^{1}(\det)$, 
		and $f_{4}$
		be the primary invariants as in \Cref{defn:primary-invariants-description}, and $R$ the $K$-subalgebra generated by them.
		Then, $S^{G} = K[f_{1}, f_{2}, f_{3}, f_{4}, h]$ for a secondary invariant $h$ of degree $q^{2}$. 
		More precisely, we have the decomposition $S^{G} = R \oplus Rh$ as $R$-modules.
		In particular, $S^{G}$ is a hypersurface with 
		Hilbert series given as
		\begin{equation*} 
			\Hilb(S^{G}, z) = 
			\frac{1 + z^{q^{2}}}
			{(1 - z)(1 - z^{2})(1 - z^{q + 1})(1 - z^{q^{2} - q})}.
		\end{equation*}
		Additionally, the invariant ring $S^{G}$ 
		does not split from $S$ (equivalently, $S^{G}$
		is not $F$-regular),
	  has $a$-invariant $-4$,
	  and is a unique factorisation domain precisely when the characteristic of $K$ is two.
		If the characteristic is odd, then 
		the class group of $S^{G}$ is $\mathbb{Z}/2$, and
		the secondary invariant $h$ can be chosen to be the Jacobian of the $f_{i}$, 
		in which case, we have $h^{2} \in R$. 
	\end{thm}

	In the other parts of the paper, we consider the conjugation action for the other classical subgroups of $\GL_{2}(K)$, 
	namely, the \emph{special linear group} $\SL_{2}(K)$ and the \emph{orthogonal group} $\OO_{2}(K)$. 
	Note that the \emph{symplectic group} $\Sp_{2}(K)$ coincides with $\SL_{2}(K)$. 
	For each of these groups, we first consider the \emph{generic} conjugation action, i.e., 
	the conjugation action on the space $\gl$ of $2 \times 2$ matrices. 
	Next, we consider the conjugation actions of $\GL_{2}$ and $\SL_{2}$ on the space $\sl$ of traceless matrices, 
	and the conjugation action of $\OO_{2}$ on the space $\oo$ of alternating matrices and on the space of symmetric matrices. 
	In particular, this includes the \emph{adjoint representation} of each classical group%
	---the action of an algebraic group on its Lie algebra, induced by conjugation. 

	\Cref{tab:summary} summarises the description of the invariant ring in each case. 
	In particular, we note that the invariant ring is never worse than a hypersurface. 
	For all the rows saying `hypersurface', the result is sharp: 
	we know that the invariant ring is not a polynomial ring; 
	in fact, its $a$-invariant is the same as that of the ambient polynomial ring. 
	The failure of being a polynomial ring can also be explained by the lack of pseudoreflections in those actions. 
	In each case, the set of primary invariants described is optimal. 
	Moreover, the additional algebra generator (when required) can be chosen as the Jacobian of our choice of primary invariants, 
	at least in odd characteristic. 
	Our choice of primary invariants for the $\GL_{2}$ and $\SL_{2}$ actions is rather uniform as \Cref{tab:summary} shows; 
	only the fourth invariant needs to be suitably modified in each case. 
	For the $\OO_{2}$ action, we have invariants of smaller degree, owing to the smaller order of the group. 
	The ring of invariants is a unique factorisation domain whenever the characteristic is two;
	the class group turns out to be either $\mathbb{Z}/2$ or $(\mathbb{Z}/2)^{2}$ whenever the ring is not a unique factorisation domain, 
	the latter appearing only for the generic action of $\OO_{2}$.

\bgroup \renewcommand{\arraystretch}{1.8}
	\begin{table}[]
	\centering
	\begin{tabular}{|c|c|c|cccc|}
	\hline
	\textbf{Group} &
	  \textbf{Space} &
	  \textbf{$\md{\pmb{K}}$} &
	  \multicolumn{1}{c|}{\textbf{Primary invariants}} &
	  \multicolumn{1}{c|}{\textbf{Type}} &
	  \multicolumn{1}{c|}{\textbf{UFD}} &
	  \textbf{Reference} \\ \hline
	\hline
	\multirow{4}{*}{$\GL_{2}(K)$} &
	  \multirow{2}{*}{$\bgeneric$} &
	  even &
	  \multicolumn{1}{c|}{\multirow{2}{*}{$\trace,\; \det,\; \mathcal{P}^{1}(\det),\; f_{4}$}} &
	  \multicolumn{1}{c|}{\multirow{2}{*}{Hypersurface}} &
	  \multicolumn{1}{c|}{Yes} &
	  \multirow{2}{*}{\Cref{mainthm}} \\ \cline{3-3} \cline{6-6}
	 &
	   &
	  odd &
	  \multicolumn{1}{c|}{} &
	  \multicolumn{1}{c|}{} &
	  \multicolumn{1}{c|}{No} &
	   \\ \cline{2-7} 
	 &
	  \multirow{2}{*}{$\btraceless$} &
	  even &
	  \multicolumn{1}{c|}{$\det,\; \mathcal{P}^{1}(\det),\; \sqrt{\widetilde{f}_{4}}$} &
	  \multicolumn{1}{c|}{Polynomial} &
	  \multicolumn{1}{c|}{Yes} &
	  \multirow{2}{*}{\Cref{thm:traceless-GL}} \\ \cline{3-6}
	 &
	   &
	  odd &
	  \multicolumn{1}{c|}{$\det,\; \mathcal{P}^{1}(\det),\; \widetilde{f}_{4}$} &
	  \multicolumn{1}{c|}{Hypersurface} &
	  \multicolumn{1}{c|}{No} &
	   \\ \hline
	\hline
	\multirow{4}{*}{$\SL_{2}(K)$} &
	  \multirow{2}{*}{$\bgeneric$} &
	  even &
	  \multicolumn{4}{c|}{Same as $\GL_{2}(K)$} \\ \cline{3-7} 
	 &
	   &
	  odd &
	  \multicolumn{1}{c|}{$\trace,\; \det,\; \mathcal{P}^{1}(\det),\; \sqrt{- f_{4}}$} &
	  \multicolumn{1}{c|}{Hypersurface} &
	  \multicolumn{1}{c|}{Yes} &
	  \Cref{thm:SL-generic-invariants} \\ \cline{2-7} 
	 &
	  \multirow{2}{*}{$\btraceless$} &
	  even &
	  \multicolumn{4}{c|}{Same as $\GL_{2}(K)$} \\ \cline{3-7} 
	 &
	   &
	  odd &
	  \multicolumn{1}{c|}{$\det,\; \mathcal{P}^{1}(\det),\; \sqrt{- \widetilde{f}_{4}}$} &
	  \multicolumn{1}{c|}{Hypersurface} &
	  \multicolumn{1}{c|}{Yes} &
	  \Cref{thm:SL-traceless-invariants} \\ \hline
	\hline
	\multirow{6}{*}{$\OO_{2}(K)$} &
	  \multirow{2}{*}{$\bgeneric$} &
	  even &
	  \multicolumn{1}{c|}{$\trace,\; \det,\; b - c,\; N(a)$} &
	  \multicolumn{1}{c|}{\multirow{2}{*}{Hypersurface}} &
	  \multicolumn{1}{c|}{Yes} &
	  \multirow{2}{*}{\Cref{thm:O-generic-invariants}} \\ \cline{3-4} \cline{6-6}
	 &
	   &
	  odd &
	  \multicolumn{1}{c|}{$\trace,\; \det,\; (b - c)^2,\; N(a)$} &
	  \multicolumn{1}{c|}{} &
	  \multicolumn{1}{c|}{No} &
	   \\ \cline{2-7} 
	 &
	  \multirow{2}{*}{$\balternating$} &
	  even &
	  \multicolumn{1}{c|}{$b$} &
	  \multicolumn{1}{c|}{\multirow{2}{*}{Polynomial}} &
	  \multicolumn{1}{c|}{\multirow{2}{*}{Yes}} &
	  \multirow{2}{*}{\Cref{rem:O-adjoint-invariants}} \\ \cline{3-4}
	 &
	   &
	  odd &
	  \multicolumn{1}{c|}{$b^{2}$} &
	  \multicolumn{1}{c|}{} &
	  \multicolumn{1}{c|}{} &
	   \\ \cline{2-7} 
	 &
	  \multirow{2}{*}{$\bsymmetric$} &
	  even &
	  \multicolumn{1}{c|}{$\trace,\; \det,\; \sum_{k} b^{2^{k}} (a + d)^{q/2 - 2^{k}}$} &
	  \multicolumn{1}{c|}{\multirow{2}{*}{Polynomial}} &
	  \multicolumn{1}{c|}{\multirow{2}{*}{Yes}} &
	  \multirow{2}{*}{\Cref{thm:O-symmetric-invariants}} \\ \cline{3-4}
	 &
	   &
	  odd &
	  \multicolumn{1}{c|}{$\trace,\; \det,\; N(a)$} &
	  \multicolumn{1}{c|}{} &
	  \multicolumn{1}{c|}{} &
	   \\ \hline
	\end{tabular}
	\caption{An overview of the descriptions of the invariant rings.}
	\label{tab:summary}
	\end{table}
\egroup

	We now compare our work with related previous works. 
	Anghel~\Cite{Anghel:SL2} considered the conjugation action of $\SL_{2}$ on $\gl$ and computed the ring of invariants for both finite and infinite fields. 
	This was done by working with three types of generators of $\SL_{2}(\mathbb{F}_{q})$, identifying invariance conditions for these, and then explicitly solving for the polynomials that satisfy those conditions. 
	In contrast, our methods are more theoretical and provide a uniform template for computing the rings of invariants for all the actions that we consider. 
	Our choice of the algebra generators also differs: Anghel constructs the third primary invariant as a product of $q+1$ linear forms, 
	whereas we give a closed form by making use of the Steenrod operations. 
	On the other hand, our fourth primary invariant is a product of linear forms, whereas Anghel's is a sum of such products. 
	Anghel also constructs the secondary invariant (in odd characteristic) as a product of suitable linear forms, whereas we construct it as the Jacobian of the primary invariants. 
	In characteristic two, both our results show the existence of the secondary invariant of the correct degree without explicitly constructing one. 
	Smith~\Cite{Smith:Conjugation} computed the rings of invariants for the conjugation action of $\GL_{2}(\mathbb{F}_{q})$ on $\gl$ and $\sl$ for $q$ an odd prime. 
	This was done by observing that the conjugation action of $\GL_{2}$ on $\sl$ is the `same' as that of $\SO_{3}$ in its defining representation. 
	Smith identifies four primary invariants and shows that the ring is a hypersurface. 
	Our choice of the fourth primary invariant differs from Smith's, and we are also able to construct a secondary invariant, 
	giving a complete algebra generating set. 
	Moreover, our techniques apply to characteristic two as well, and we obtain the corresponding structural results. 
	While the previous works were mainly concerned with computing the rings of invariants, 
	we determine further algebraic properties of $S^{G}$ such as its $a$-invariant, its class group, and show that the inclusion $S^{G} \into S$ does not split unless $S^{G}$ is a polynomial ring. 

	The terms ``primary'' and ``secondary'' above have their usual meanings as in the invariant theory of finite groups, that we elaborate on now. 
	A good reference for this is~\Cite[\S3]{DerksenKemper}. 
	If $G$ is a finite group acting on a polynomial ring $S = K[x_{1}, \ldots, x_{n}]$ by degree-preserving $K$-algebra automorphisms, then \deff{primary invariants} are homogeneous invariants $f_{1}, \ldots, f_{n} \in S^{G}$ satisfying any of the following equivalent properties:
	\begin{enumerate}[label=(\alph*)]
		\item The radical of $(f_{1}, \ldots, f_{n}) S$ is the homogeneous maximal ideal of $S$, i.e., the $f_{i}$ form a \deff{homogeneous system of parameters} for $S$.
		\item The set of common zeroes of $\{f_{1}, \ldots, f_{n}\}$ in $\overline{K}^{n}$ is the origin, where $\overline{K}$ is the algebraic closure of $K$.
		\item The subalgebra $R \coloneqq K[f_{1}, \ldots, f_{n}]$ is a polynomial ring and $S^{G}$ is a finite $R$-module. 
		In this case, we refer to $R$ as a \deff{Noether normalisation} for $S^{G}$.
	\end{enumerate}

	Once we have a set of primary invariants with corresponding Noether normalisation $R$, the next objective is to determine $R$-module generators for $S^{G}$. 
	Finiteness tells us that we may (minimally) write
	\begin{equation} \label{eq:secondary-invariants-decomposition}
		S^{G} = R h_{1} + \cdots + R h_{s}
	\end{equation}
	for some $h_{j} \in S^{G}$, called the \deff{secondary invariants}. 
	The above sum is direct precisely when the ring $S^{G}$ is Cohen--Macaulay.
	In any case, we then obtain the equality of $K$-algebras
	\begin{equation*} 
		S^{G} = K[f_{1}, \ldots, f_{n}, h_{1}, \ldots, h_{s}].
	\end{equation*}
	We may always assume $h_{1} = 1$ and exclude it from the above algebra generating set.
	We remark that primary and secondary invariants are not uniquely determined;
	moreover, the minimal number of secondary invariants depends on the choice of primary invariants. 
	A lower bound is given in~\Cite[Theorem 3.7.1]{DerksenKemper}; the equality holds precisely when $S^{G}$ is Cohen--Macaulay. 
	We will see (\Cref{rem:optimal}) that the choices of the primary and secondary invariants in \Cref{tab:summary} are optimal. 

	The outline of the proof of the main result---and of \Cref{part:one}---is as follows: 
	We first introduce the relevant notation and definitions.
	In particular, we introduce a `larger' group $\Gamma$ acting on $S$ such that $S^{\Gamma} \subset S^{G}$. 
	We collect some facts about the conjugation action in \Cref{sec:facts-conjugation} that only rely on basic linear algebra. 
	In \Cref{sec:primary-invariants}, we define the invariants 
	$f_{1}, f_{2}, f_{3}, f_{4}$ 
	and show that these form a set of primary invariants; we denote the corresponding Noether normalisation as $R$. 
	We then deduce homological properties of $S^{G}$ and $S^{\Gamma}$ in \Cref{sec:homological-properties}. 
	In particular, we show that $S^{G}$ is a hypersurface of the form $R \oplus R \eta$, and that $S^{\Gamma}$ is the polynomial ring $R$. 
	We make use of the $a$-invariant to obtain the degree of $\eta$ to be $q^{2}$. 
	Consequently, we obtain the Hilbert series of $S^{G}$ and conclude that any element of $S^{G} \mysetminus R$ of the correct degree can serve as $\eta$. 
	In \Cref{sec:missing-invariant}, we construct an invariant $h$ of the correct degree, by defining it to be the Jacobian of the primary invariants. 
	We make use of the $\Gamma$-action to show that $h^{2} \in R$, and that $h \notin R$ when $\chr(K)$ is odd; 
	this finishes the problem of describing the generators and relations. 
	The additional results about $F$-regularity and factoriality are proven in \Cref{sec:additional}. 
	To construct one of the primary invariants, we make use of the Steenrod operations;
	the relevant notation and results are reviewed in \Cref{sec:steenrod}. 
	\Cref{part:two} concerns the action of $\GL_{2}$ on the space $\sl$ of traceless matrices, as well as the actions of $\SL_{2}$ on each of $\gl$ and $\sl$. 
	In \Cref{part:three}, we consider the actions of $\OO_{2}$ on 
	the space $\oo$ of alternating matrices, 
	the space of symmetric matrices, 
	and 
	$\gl$. 
	The proof in each case follows the same template as that for \Cref{mainthm}. 

	\subsection*{Acknowledgements}
	The author thanks Anurag K. Singh for several interesting discussions. 
	\Cref{part:two,part:three} arose out of the natural questions suggested by 
	Karthik Ganapathy, Suhas Gondi, Vaibhav Pandey, Steven Sam, and Ramanujan Srihari; 
	the author thanks them for the same. 
	The author is also grateful to the anonymous referee for the numerous suggestions and references to the existing literature. 
	This work has greatly benefited from examples computed using the computer algebra system \texttt{Magma}~\Cite{Magma}, the use of which is gratefully acknowledged.

\part{The action of \texorpdfstring{$\GL_{2}$}{GL2} on \texorpdfstring{$\gl$}{gl2} } \label{part:one}

	Let $q$ be a power of a positive prime $p$. 
	We set $K \coloneqq \mathbb{F}_{q}$, the field with $q$ elements, and let $\overline{K}$ denote its algebraic closure. 
	Let $G \coloneqq \GL_{2}(K)$ be the general linear group, and 
	$S \coloneqq K[a, b, c, d]$ the polynomial ring over $K$ in four variables. 
	We consider the conjugation action of $G$ on $S$ given by
	\begin{equation} \label{eq:G-action-defn}
		\sigma \colon
		\begin{bmatrix}
			a & b \\
			c & d \\
		\end{bmatrix}
		\mapsto
		\sigma^{-1}
		\begin{bmatrix}
			a & b \\
			c & d \\
		\end{bmatrix}
		\sigma,
	\end{equation}
	i.e., $\sigma \in G$ maps each variable to the corresponding entry of 
	$\sigma^{-1}
		\smatrix{
			a & b \\
			c & d \\
		}
		\sigma$.

	We note that the above action comes from the conjugation action of $G$ on $V \coloneqq \MM_{2}(K)$, the space of $2 \times 2$ matrices, as follows: 
	the action of $G$ on $V$ is given by $(\sigma, M) \mapsto \sigma M \sigma^{-1}$. 
	Said otherwise, this is the adjoint representation of $G$ on $V = \gl$,
	which, in turn, gives us a left action of $G$ on $\Sym(V^{\ast}) \cong S$. 
	The isomorphism is constructed as follows: let $E_{ij} \in V$ denote the matrix whose sole nonzero entry is a $1$ in the $(i, j)$ position.
	Then, $B \coloneqq \{E_{11}, E_{12}, E_{21}, E_{22}\}$ is a basis for $V$. 
	If we let $\{a, b, c, d\} \subset V^{\ast}$ denote its dual basis, then $\Sym(V^{\ast})$ is the polynomial ring $S = K[a, b, c, d]$, and the action is precisely~\Cref{eq:G-action-defn} under this identification.
	We shall use both perspectives in describing the ring of invariants $S^{G}$. 

	We denote the corresponding representation as 
	\begin{equation} \label{eq:representation-rho}
		\rho \colon G \to \GL(V),
	\end{equation}
	and set $\widehat{G} \coloneqq \rho(G) \le \GL(V)$. 
	The group $\widehat{G}$ acts faithfully on both $V$ and $S$ via the action of $G$.
	While we have the equality $S^{G} = S^{\widehat{G}}$, the distinction between $G$ and $\widehat{G}$ is necessary when we use results that require that the group action be faithful.
	Note that the action of $G$ is not faithful, i.e., $\rho$ is not injective: 
	the kernel consists precisely of the invertible scalar matrices, and thus $\widehat{G} \cong \PGL_{2}(K)$. 
	In particular, the orders of the groups are 
	\begin{equation*} 
		\md{G} = (q^{2} - 1)(q^{2} - q) \andd \mdd{\widehat{G}} = q(q^{2} - 1).
	\end{equation*}
	Thus, the action of $G$ is \emph{modular}, i.e., $\md{\rho(G)}$ is divisible by the characteristic of $K$. 

	We define the larger group $\Gamma \coloneqq \langle \widehat{G}, \tau_{ad} \rangle \le \GL(V)$, where $\tau_{ad} \in \GL(V)$ is the automorphism fixing $E_{12}$ and $E_{21}$, and swapping $E_{11} \leftrightarrow E_{22}$.  
	Equivalently, the action of $\tau_{ad}$ on $S$ is given by fixing $b$ and $c$, and swapping $a \leftrightarrow d$. 

	As the field $K$ is finite, there are monic irreducible polynomials in $K[x]$ of any given positive degree. 
	We fix an irreducible quadratic polynomial $g(x) = x^{2} - \tau x + \delta \in K[x]$ for the rest of \Cref{part:one}.

	The ring $S$ will have its standard $\mathbb{N}$-grading throughout the paper, i.e., all the variables have degree~$1$. 
	The subrings that we construct will be graded subrings that are finitely generated algebras over $K$. 
	For such a graded ring $R = \bigoplus_{n \ge 0} R_{n}$, its \emph{Hilbert series} is the power series
	\begin{equation*} 
		\Hilb(R, z) 
		\coloneqq 
		\sum_{n \ge 0} \rank_{K}(R_{n}) z^{n} 
		\in \mathbb{Q}\llbracket z \rrbracket.
	\end{equation*}
	We suppress the variable when no explicit mention is necessary.
	The Hilbert--Serre theorem asserts that the above power series is a rational function, see~\Cite[Theorem 11.1]{AtiyahMacdonald}. 
	Writing $\Hilb(R) = f/g$ for polynomials $f$ and $g$, we define the \deff{degree} of $\Hilb(R)$ to be the difference 
	$\deg \Hilb(R) \coloneqq \deg(f) - \deg(g)$.

\section{Preliminaries about the conjugation action} \label{sec:facts-conjugation}
			
	We collect some basic facts pertaining to the linear algebra of the conjugation action. 
	In particular, we show that the representation $\rho$ factors through $\SL(V)$ and contains no pseudoreflections. 
	Note that the (non-canonical) isomorphism 
	$(-)^{\ast} \colon V \to V^{\ast}$ 
	of $K$-vector spaces obtained by our choice of bases is given as 
	\begin{equation*} 
		\begin{bmatrix}A & B \\ C & D\end{bmatrix} 
		\mapsto A a + B b + C c + D d.
	\end{equation*} 
	While this isomorphism is not $G$-equivariant, the following lemma describes how it interacts with $G$.

	\begin{lem}
		For $v \in V$ and $\sigma \in G$, we have 
		$(\sigma \cdot v)^{\ast} = (\sigma^{-1})^{\tr} \cdot v^{\ast}$. 
	\end{lem}
	\begin{proof} 
		For this proof, we use $\rho$ and $\rho^{\ast}$ to denote the respective homomorphisms $G \to \GL_{4}(K)$ obtained by the identifications $V \cong K^{4}$ and $V^{\ast} \cong K^{4}$ using our choice of standard bases. 
		By~\Cite[Lemma 1.1.1]{CampbellWehlau:ModularInvariantTheory}, we have $\rho^{\ast}(\sigma) = \rho(\sigma^{^{-1}})^{\tr}$. 
		Thus, it suffices to show $\rho(\sigma)^{\tr} = \rho(\sigma^{\tr})$. 
		To this end, note that the standard inner product on $K^{4}$ translates to one on $V$ as $\langle M, N \rangle = \trace(M N^{\tr})$ for $M, N \in V$. 
		Then, for $\sigma \in G$, we have
		\begin{align*} 
			\langle \sigma \cdot M, N \rangle 
			= \trace(\sigma M \sigma^{-1} N^{\tr}) 
			= \trace(M \sigma^{-1} N^{\tr} \sigma) 
			= \langle M, \sigma^{\tr} \cdot N \rangle.
		\end{align*}
		Thus, 
		$\langle \rho(\sigma)M, N \rangle = \langle M, \rho(\sigma^{\tr}) N \rangle$, 
		giving us the desired assertion that $\rho(\sigma)^{\tr} = \rho(\sigma^{\tr})$.
	\end{proof}

	\begin{lem} 
		Let $\Omega \subset V$ be the set of matrices whose characteristic polynomial is equal to $g(x)$. 
		Then, we have
		\begin{equation*} 
			\Omega =
			\left\{
			\begin{bmatrix}
				A & B \\
				-\frac{g(A)}{B} & \tau - A
			\end{bmatrix}
			:
			A \in K, 
			B \in K^{\times}
			\right\}.
		\end{equation*}
		Moreover, any two elements of $\Omega$ are similar, and thus $\Omega$ is a single orbit for the conjugation action of $G$ on $V$.
	\end{lem}
	\begin{proof} 
		Finite fields are perfect, and hence $g(x)$ has distinct factors $\mu, \nu \in \overline{K}$. 
		In turn, any $M \in \Omega$ is similar to $\diag(\mu, \nu)$ over $\overline{K}$. 
		In particular, any two elements of $\Omega$ are similar over $\overline{K}$, implying the same over $K$.
		
		We now determine $\Omega$.
		Consider an arbitrary matrix
		\begin{equation*} 
			M = 
			\begin{bmatrix}
				A & B \\
				C & D \\
			\end{bmatrix}
			\in \MM_{2}(K).
		\end{equation*}
		For its characteristic polynomial to be $g(x) = x^{2} - \tau x + \delta$, we must have $\trace(M) = \tau$ and $\det(M) = \delta$. 
		The first condition gives us $D = \tau - A$ and in turn
		\begin{equation*} 
			A(\tau - A) - B C = \delta.
		\end{equation*}
		Rearranging the above give us
		\begin{equation*} 
			- B C = g(A).
		\end{equation*}
		Because $g$ is irreducible, we have that $g(A) \neq 0$ for all $A \in K$, and thus $B$ is nonzero. 
		We may then solve for $C$ to obtain the desired result.
	\end{proof}

	\begin{cor} \label{cor:orbit-irreducible-quadraic}
		The set of homogeneous linear polynomials in $K[a, b, c, d]$ defined by
		\begin{equation*} 
			\Omega \coloneqq \left\{Aa + Bb - \frac{g(A)}{B}c + (\tau - A)d 
			:
			A \in K,\,
			B \in K^{\times}
			\right\}
		\end{equation*}
		forms a single $G$-orbit of size $q^{2} - q$. 
		Moreover $\Omega$ is stable under the action of $\Gamma$.
	\end{cor}
	\begin{proof} 
		Only the last statement needs a proof. 
		To this end, note that if $A \in K$ and $B \in K^{\times}$, then
		\begin{align*} 
			\tau_{ad}\left(Aa + Bb - \frac{g(A)}{B}c + (\tau - A)d\right) 
			&= (\tau - A)a + Bb - \frac{g(A)}{B}c + Ad.
		\end{align*}
		Because $g(A) = g(\tau - A)$, we see that the above is indeed an element of $\Omega$.
	\end{proof}

	\begin{lem} \label{lem:action-factors-through-SL}
		We have the inclusion $\widehat{G} \subset \SL(V)$ as subgroups of $\GL(V)$.
	\end{lem}
	\begin{proof} 
		We wish to show that $\rho(G) \subset \SL(V)$. 
		The representation $\rho$ is given by $\rho(\sigma)(M) = \sigma M \sigma^{-1}$ for $\sigma \in G$ and $M \in V$.
		Said otherwise, 
		$\rho(\sigma) = L(\sigma) \circ R(\sigma)^{-1}$, 
		where $L(\sigma)$ and $R(\sigma)$ denote left and right multiplication by $\sigma$, respectively. 
		Thus, $\det(\rho(\sigma)) = \det(L(\sigma))/\det(R(\sigma))$ and
		it suffices to show that 
		$\det(L(\sigma)) = \det(R(\sigma))$. 
		This is now a simple linear algebra exercise, see for example,~\Cite[\S5.4 Exercise 12]{HoffmanKunze}.
	\end{proof}

	\begin{lem} \label{lem:G-is-small}
		The action of $G$ contains no pseudoreflections, i.e., $\rank(\rho(\sigma) - \id) \neq 1$ for all $\sigma \in G$.
	\end{lem}
	\begin{proof} 
		In view of the rank-nullity theorem, we wish to show that if $\sigma \in G$, then $\ker(\rho(\sigma) - \id)$
		has dimension different from~$3$. 
		The kernel consists precisely of those $M \in \MM_{2}(K)$ that commute with $\sigma$. 
		Because this dimension does not change upon enlarging the base field, we may assume that $\sigma$ is in Jordan form. 
		Considering the possibilities for a $2 \times 2$ Jordan form, one sees that the requisite dimension is either~$2$ or~$4$.
	\end{proof}

	We now analyse the action of $G$ on $V$ via its Sylow-$p$ subgroups. 
	To this end, consider the \emph{unipotent} group 
	\begin{equation} \label{eq:unipotent-defn}
		P \coloneqq 
		\begin{bmatrix}
			1 & K \\
			0 & 1 \\
		\end{bmatrix}
		=
		\left\{
		\begin{bmatrix}
			1 & \alpha \\
			0 & 1 \\
		\end{bmatrix}
		: 
		\alpha \in K
		\right\}.
	\end{equation}

	\begin{lem} \label{lem:P-fixed-codim-2}
		We have $\dim(V^{P}) = 2$,
		equivalently, $\codim(V^{P}) = 2$. 
	\end{lem}
	\begin{proof} 
		It is a straightforward computation that if $\alpha \in K^{\times}$, then the matrices commuting with $\smatrix{1 & \alpha \\ 0 & 1}$ are precisely those of the form $\smatrix{A & B \\ 0 & A}$ with $A, B \in K$. 
		In turn, $V^{P}$ is the two-dimensional space $\left\{\smatrix{A & B \\ 0 & A} : A, B \in K\right\}$. 
		The statement about the codimension follows after noting that $\dim(V) = 4$. 
	\end{proof}

	\begin{lem} \label{lem:commutator}
		If $K \neq \mathbb{F}_{2}$, then $\widehat{G}/[\widehat{G}, \widehat{G}] \cong K^{\times}/(K^{\times})^{2}$.
	\end{lem}
	\begin{proof} 
		The commutator subgroup $[\GL_{n}(K), \GL_{n}(K)]$ is equal to $\SL_{n}(K)$ except when $n = 2$ and $K = \mathbb{F}_{2}$; 
		see~\Cite{Thompson:Commutators} or~\Cite[\S6.7]{Jacobson:BAI}.
		In turn, we get $[\PGL_{n}(K), \PGL_{n}(K)] = \PSL_{n}(K)$.
		The result now follows from 
		the fact that $\widehat{G} \cong \PGL_{2}(K)$
		and 
		the exact sequence of groups 
		\begin{equation*} 
			1 \to \PSL_{n}(K) \into \PGL_{n}(K) \xrightarrow{\det} K^{\times}/(K^{\times})^{n} \to 0. \qedhere
		\end{equation*}
	\end{proof}

\section{The primary invariants} \label{sec:primary-invariants}
	
	We now describe a set of primary invariants for the conjugation action.
	Two natural candidates to start with are the trace and determinant. 
	As we are working over finite fields, we may use the \emph{Steenrod operations} $\mathcal{P}$ (see \Cref{sec:steenrod}) to produce a third new invariant from these.
	The fourth invariant is a certain orbit product.
	
	\begin{defn} \label{defn:primary-invariants-description}
		The elements $f_{1}, f_{2}, f_{3}, f_{4} \in S$ are defined as
		\begin{equation} \label{eq:primary-invariants-description}
			\begin{aligned}
				f_{1} &\coloneqq a + d, \\
				f_{2} &\coloneqq a d - b c, \\
				f_{3} &\coloneqq a^{q} d + a d^{q} - b^{q} c - b c^{q}, \\
				f_{4} &\coloneqq 
				\prod_{\substack{A \in K \\ B \in K^{\times}}} 
				\left(A a + B b - \frac{g(A)}{B} c + (\tau - A) d\right). \\
			\end{aligned}
		\end{equation}
	\end{defn}

	The elements above are readily seen to be homogeneous of degrees $1$, $2$, $q + 1$, and $q^{2} - q$, respectively.
	We set $R \coloneqq K[f_{1}, f_{2}, f_{3}, f_{4}] \subset S$.
	We will show that $R$ is a Noether normalisation for $S^{G}$ and that we have a decomposition of the form $S^{G} = R \oplus R h$. 
	In particular, $S^{G}$ is a hypersurface. 
	Along the way, we will also show that $R = S^{\Gamma}$, i.e., $R$ is itself a ring of invariants. 
	This equality shows that while $f_{4}$ depends on the choice of the irreducible quadratic $g(x)$, the Noether normalisation $R$ does not.

	\begin{prop} \label{prop:primary-invariants-are-invariant}
		We have $R \subset S^{\Gamma} \subset S^{G}$.
	\end{prop}
	\begin{proof} 
		The elements $f_{1}$ and $f_{2}$ are the trace and determinant of $\smatrix{a & b \\ c & d}$ and are hence invariant under conjugation. 
		These are also symmetric in $a$ and $d$ and hence $\Gamma$-invariant.
		The element $f_{3}$ is $\mathcal{P}^{1}(f_{2})$ and hence is $\Gamma$-invariant in view of \Cref{lem:steenrod-invariant}. 
		The invariance of $f_{4}$ follows from \Cref{cor:orbit-irreducible-quadraic}. 
	\end{proof}

	We do our only gritty calculation below to show that the $f_{i}$ form a homogeneous system of parameters for $S$; 
	in particular, $R$ is a Noether normalisation for each of $S^{G}$ and $S^{\Gamma}$. 

	\begin{thm} \label{thm:primary-invariants}
		The invariants $f_{1}, f_{2}, f_{3}, f_{4}$ 
		form a homogeneous system of parameters for $S$.
	\end{thm}
	\begin{proof} 
		It suffices to show that the only solution in $\overline{K}^{4}$ to $f_{1} = f_{2} = f_{3} = f_{4} = 0$ is the origin. 
		Let $(a, b, c, d) \in \overline{K}^{4}$ be such a common solution. 
		We immediately discard the equation $f_{1} = 0$ by substituting $d = -a$ in the other equations. 
		The equation $f_{4} = 0$ tells us that there exist $A \in K$, $B \in K^{\times}$ such that
		\begin{equation} \label{eq:b-in-terms-a-c}
			b = \frac{\tau - 2A}{B}a + \frac{g(A)}{B^{2}}c.
		\end{equation}

		After the substitution $d = -a$, the equation $f_{2} = 0$ gives us $a^{2} + b c = 0$. Using~\Cref{eq:b-in-terms-a-c}, we obtain
		\begin{equation*} 
			a^{2} + \frac{\tau - 2A}{B} ac + \frac{g(A)}{B^{2}} c^{2} = 0.
		\end{equation*}

		Writing $g(A) = A^{2} - \tau A + \delta$, the above can be rearranged to get
		\begin{equation*} 
			\left(\frac{A}{B}c - a\right)^{2} 
			- \tau \left(\frac{A}{B}c - a\right)\left(\frac{c}{B}\right) 
			+ \delta \left(\frac{c}{B}\right)^{2} = 0.
		\end{equation*}	

		We may factor $g(x)$ over $\overline{K}$ to get $g(x) = (x + \mu)(x + \nu)$ for some $\mu, \nu \in \overline{K} \mysetminus K$. 
		In turn, the above factors as
		\begin{equation*} 
			\left(\frac{A}{B}c - a + \mu \frac{c}{B}\right)
			\left(\frac{A}{B}c - a + \nu \frac{c}{B}\right)
			= 0.
		\end{equation*}

		Without loss of generality, we may assume that the first factor is zero, giving us
		\begin{equation*} 
			a = \frac{A + \mu}{B}c.
		\end{equation*}

		Substituting this in~\Cref{eq:b-in-terms-a-c} and using $\mu + \nu = -\tau$ gives us
		\begin{align*} 
			b 
			&= \frac{\tau - 2A}{B}\cdot\frac{A + \mu}{B}c + \frac{g(A)}{B^{2}}c \\
			&= \frac{-\mu - \nu - 2A}{B}\cdot\frac{A + \mu}{B}c + \frac{(A + \mu)(A + \nu)}{B^{2}} c \\
			&= -\left(\frac{A + \mu}{B}\right)^{2}c.
		\end{align*}
		Letting $\gamma \coloneqq (A + \mu)/B$, we see that 
		\begin{equation} \label{eq:a-b-in-terms-of-c}
			a = \gamma c \andd b = -\gamma^{2} c.
		\end{equation}
		Note that $\mu \notin K$, and hence $\gamma \notin K$ as well. 

		We now substitute the above in $f_{3} = 0$. We get
		\begin{equation*} 
			(-2\gamma^{q + 1} + \gamma^{2q} + \gamma^{2}) c^{q + 1} = 0.
		\end{equation*}
		
		The above factors as
		\begin{equation*} 
			(\gamma^{q} - \gamma)^{2} \cdot c^{q + 1} = 0.
		\end{equation*}
		The first term is nonzero because $\gamma \notin K$, and hence $\gamma$ cannot be a root of the polynomial $x^{q} - x$. Thus, we get $c = 0$ and in view of~\Cref{eq:a-b-in-terms-of-c}, we are done.
	\end{proof}

\section{Homological properties} \label{sec:homological-properties}
	
	In this section, 
	we show that $S^{\Gamma}$ is equal to the Noether normalisation $R$ defined earlier, and that
	$S^{G}$ is a hypersurface that decomposes as $R \oplus R \eta$. 
	We begin by showing that $S^{G}$ is a Cohen--Macaulay ring.

	\begin{prop} \label{prop:SG-cohen-macaulay}
		The ring $S^{G}$ is Cohen--Macaulay.
	\end{prop}
	\begin{proof} 
		Recall the unipotent subgroup $P \coloneqq \smatrix{1 & K \\ 0 & 1}$. 
		By \Cref{lem:P-fixed-codim-2}, we have $\codim(V^{P}) = 2$, and in turn, $S^{P}$ is Cohen--Macaulay by~\Cite[Theorem 3.9.2]{CampbellWehlau:ModularInvariantTheory}. 
		Note that $[G : P] = (q^{2} - 1)(q - 1)$ is invertible in $K$, and thus the inclusion $S^{G} \into S^{P}$ splits. 
		Indeed, the \emph{relative transfer} map
		\begin{align*} 
			S^{P} &\to S^{G} \\
			s &\mapsto \frac{1}{[G : P]} \sum_{\sigma P \in G/P} \sigma(s)
		\end{align*}
		defines an $S^{G}$-linear splitting.
		Because this is a finite extension, we get that $S^{G}$ is Cohen--Macaulay.
	\end{proof}

	\begin{por} \label{por:conjugation-invariants-CM}
		Let $H \le \GL_{2}(K)$ be any subgroup acting via conjugation on $S$. 
		Then, $S^{H}$ is Cohen--Macaulay.
	\end{por}
	\begin{proof} 
		Let $Q$ be a Sylow-$p$ subgroup of $H$. 
		It suffices to show that $S^{Q}$ is Cohen--Macaulay. 
		Upon conjugation, we may assume that $Q \le P$. 
		But then, $\codim(V^{Q}) \le \codim(V^{P}) = 2$, and thus $S^{Q}$ is Cohen--Macaulay. 
	\end{proof}

	\begin{prop} \label{prop:hypersurface}
		The ring $S^{G}$ is a hypersurface. 
		Specifically, if $R$ is the Noether normalisation as before, then there exists an invariant $\eta \in S^{G}$ such that 
		$S^{G} = R \oplus R \eta$ as $R$-modules,
		and hence
		$S^{G} = K[f_{1}, f_{2}, f_{3}, f_{4}, \eta]$ as $K$-algebras. 
		In particular,
		\begin{equation} \label{eq:hilbert-series-hypersurface}
			\Hilb(S^{G}, z) = 
			\frac{1 + z^{\deg \eta}}
			{(1 - z)(1 - z^{2})(1 - z^{q + 1})(1 - z^{q^{2} - q})}.
		\end{equation}
	\end{prop}
	\begin{proof} 
		Because $S^{G}$ is Cohen--Macaulay,
		we may use~\Cite[Theorem 3.7.1]{DerksenKemper} to determine the number of minimal secondary invariants (with respect to the $f_{i}$) as
		\begin{equation*} 
			\frac{\prod_{i = 1}^{4} \deg(f_{i})}{\mdd{\widehat{G}}} = \frac{1 \cdot 2 \cdot (q + 1) \cdot (q^{2} - q)}{q (q^{2} - 1)} = 2,
		\end{equation*}
		where $\widehat{G}$ is the image of $\rho$ defined in~\Cref{eq:representation-rho}. 
		As $1 \in S^{G}$ is always a minimal secondary invariant, the other secondary invariant is the $\eta$ as in the statement. 
		The Hilbert series follows by our knowledge of the degrees of the $f_{i}$ and the fact that $R$ is a polynomial algebra on the $f_{i}$.
	\end{proof}

	\begin{cor} \label{cor:R-is-S-Gamma}
		We have the equality $R = S^{\Gamma}$. 
		In particular, $S^{\Gamma}$ is a polynomial ring, and $R$ is independent of the choice of the irreducible quadratic $g(x)$.
	\end{cor}
	\begin{proof} 
		We have the integral extensions of normal domains $R \subset S^{\Gamma} \subsetneq S^{G}$; 
		to see that the latter inclusion is strict, note that $\Gamma \neq \widehat{G}$ because $\widehat{G}$ contains no pseudoreflections (\Cref{lem:G-is-small}) whereas $\Gamma$ contains the pseudoreflection $\tau_{ad}$.
		By \Cref{prop:hypersurface}, the degree of the extension $R \subset S^{G}$ is two, forcing $R = S^{\Gamma}$.
	\end{proof}

	We now calculate the $a$-invariant of $S^{G}$ and use it to determine the ring of invariants. 
	For an introduction to the $a$-invariant, we refer the reader to~\Cite{GotoWatanabe:GradedI, BrunsHerzog:a-invariants}. 
	For a graded Cohen--Macaulay ring, the $a$-invariant is simply the degree of the Hilbert series.

	\begin{prop} \label{prop:a-invariant-invariant}
		We have $a(S^{G}) = a(S) = -4$.
	\end{prop}
	\begin{proof} 
		By \Cref{lem:action-factors-through-SL,lem:G-is-small}, we know that $\widehat{G}$ is a subgroup of $\SL(V)$ that contains no pseudoreflections.
		This result now follows from~\Cite[Theorem 4.4]{GoelJeffriesSingh}.
	\end{proof}

	\begin{por} \label{por:conjugation-invariants-a-invariant}
		Let $H \le \GL_{2}(K)$ be any subgroup acting via conjugation on $S$. 
		Then, $a(S^{H}) = a(S) = -4$. 
		\hfill $\square$
	\end{por}

	\begin{cor} \label{cor:hilbert-series-SG}
		The Hilbert series of $S^{G}$ is given as
		\begin{equation} \label{eq:hilbert-series-SG}
			\Hilb(S^{G}, z) = \frac{1 + z^{q^{2}}}{(1 - z)(1 - z^{2})(1 - z^{q + 1})(1 - z^{q^{2} - q})}.
		\end{equation}
	\end{cor}
	\begin{proof} 
		By \Cref{prop:SG-cohen-macaulay}, $S^{G}$ is Cohen--Macaulay, and thus the equality $a(S^{G}) = \deg \Hilb(S^{G})$ holds. 
		Using this, we may solve for $\deg \eta$ in~\Cref{eq:hilbert-series-hypersurface} and get the desired equality.
	\end{proof}

	\begin{cor} \label{cor:q2-degree-invariant-does-the-job}
		There exists an invariant $\eta \in S^{G} \mysetminus R$ of degree $q^{2}$. 
		For any such $\eta$, we have the $R$-module decomposition $S^{G} = R \oplus R \eta$. 
		In turn, we have $S^{G} = K[f_{1}, f_{2}, f_{3}, f_{4}, \eta]$.
	\end{cor}
	\begin{proof} 
		The existence of such an $\eta$ follows from the knowledge of the Hilbert series. Indeed, we have
		\begin{equation*} 
			\Hilb(S^{G}, z) - \Hilb(R, z) = \frac{z^{q^{2}}}{(1 - z)(1 - z^{2})(1 - z^{q + 1})(1 - z^{q^{2} - q})}.
		\end{equation*}
		The coefficient of $z^{q^{2}}$ in the above is $1$, proving the first statement. 
		For the second statement, we note that $\eta$ is then of minimal degree.
	\end{proof}

	\begin{rem} \label{rem:optimal}
		We remark that our choices of invariants have been optimal in the following ways:
		\begin{enumerate}[label=(\alph*)]
			\item Because the ring $S^{G}$ is a four-dimensional ring that is not a polynomial ring, we need at least~$5$ algebra generators for $S^{G}$, which is what we have obtained.
			\item The primary invariants are optimal if we are trying to minimise the product of their degrees. 
			Indeed, by~\Cite[Proposition 3.3.5, Theorem 3.7.1]{DerksenKemper}, 
			the product of degrees is a multiple of $\mdd{\widehat{G}}$,
			with the product being equal to $\mdd{\widehat{G}}$ only if $S^{G}$ is a polynomial ring. 
			Thus, in our situation, the smallest product that one can obtain is~$2\mdd{\widehat{G}}$, which is what we do.
			Similarly, we see that the number of secondary invariants (with respect to any set of primary invariants) must be at least~$2$, and we obtain this minimum.
		\end{enumerate}
	\end{rem}

\section{The missing invariant} \label{sec:missing-invariant}

	We now show that the secondary invariant $\eta$ may be chosen as the Jacobian of the $f_{i}$, when the characteristic is odd. 
	To this end, we define
	\begin{align*} 
		h &\coloneqq \Jac(f_{1}, f_{2}, f_{3}, f_{4}) \\
		&= 
		\det
		\begin{bmatrix}
			1 & 0 & 0 & 1 \\
			d & -c & -b & a \\
			d^{q} & -c^{q} & -b^{q} & a^{q} \\
			\frac{\partial f_{4}}{\partial a} & \frac{\partial f_{4}}{\partial b} & \frac{\partial f_{4}}{\partial c} & \frac{\partial f_{4}}{\partial d} \\
		\end{bmatrix}.
	\end{align*}
	The element $h$ is readily seen to be homogeneous of degree $q^{2}$ with the caveat that $h$ may be zero. 
	We first prove that this is not the case. 

	\begin{lem} 
	 	The element $h$ is nonzero.	
	\end{lem} 
	\begin{proof} 
		By~\Cite[Proposition~5.4.2]{Benson:PolynomialInvariantsBook}, it suffices to show that the field extension $K(f_{1}, f_{2}, f_{3}, f_{4}) \subset K(a, b, c, d)$ is finite and separable. 
		In view of \Cref{cor:R-is-S-Gamma}, this field extension is precisely $\Frac(S)^{\Gamma} \subset \Frac(S)$, and thus is a finite Galois extension.
	\end{proof}

	\begin{prop}
		We have $h \in S^{G}$ and $h^{2} \in R$. 
		If the characteristic of $K$ is odd, then $h \notin R$. 
	\end{prop}
	\begin{proof} 
		Because $h$ is the Jacobian of $G$-invariant elements and the action of $G$ factors through $\SL(V)$ (by \Cref{lem:action-factors-through-SL}), the chain rule yields $h \in S^{G}$, see~\Cite[Proposition 1.5.6]{Smith:PolynomialInvariantsBook}. 

		For the remaining statements, note that by \Cref{cor:R-is-S-Gamma}, we have $R = S^{\Gamma} = S^{G} \cap S^{\tau_{ad}}$. 
		Thus, it suffices to show that $h^{2}$ is $\tau_{ad}$-invariant and that $h$ is not $\tau_{ad}$-invariant in odd characteristic. 
		The automorphism $\tau_{ad}$ switches the extreme columns of $J$, giving us $\tau_{ad} \cdot h = -h$ and $\tau_{ad} \cdot h^{2} = h^{2}$. 
		This finishes the proof.
	\end{proof}

	\begin{rem}
		The same proof above shows that if $\chr(K) = 2$, then $h$ is indeed an element of $R$. 
		Therefore, we need to pick a different element of degree $q^{2}$ to generate the invariant ring,
		the existence of such an element being ensured by \Cref{cor:q2-degree-invariant-does-the-job}.
	\end{rem}

	Thus, for odd characteristic, $h$ fulfils the hypothesis of $\eta$ as in \Cref{cor:q2-degree-invariant-does-the-job}, giving us the corresponding statements of \Cref{mainthm}.

\section{Additional properties of the invariant ring} \label{sec:additional}

	We now prove the additional ring-theoretic properties of the invariant ring mentioned in \Cref{mainthm}. 

	A natural question to ask is whether the inclusion $S^{G} \into S$ splits $S^{G}$-linearly. 
	Because this is a finite extension and $S$ is a polynomial ring, this question is equivalent to asking whether $S^{G}$ is \emph{$F$-regular}; 
	indeed, a direct summand of a polynomial ring is $F$-regular, and conversely, an $F$-regular ring splits off from any finite extension, see~\Cite[Theorem 5.25]{HochsterHuneke:JAG}.

	\begin{prop} \label{prop:SG-not-F-regular}
		The inclusion $S^{G} \into S$ does not split $S^{G}$-linearly. Equivalently, the ring $S^{G}$ is not $F$-regular.
	\end{prop}
	\begin{proof} 
		By \Cref{prop:a-invariant-invariant}, the $a$-invariants of $S$ and $S^{G}$ are the same. 
		Because the action of $G$ is modular,	we get the result by~\Cite[Theorem 2.18]{Jeffries:Thesis} or~\Cite[Corollary 4.2]{GoelJeffriesSingh}.

		Alternatively, by \Cref{lem:G-is-small}, the action of $G$ is modular and contains no pseudoreflections. 
		The result then follows from~\Cite[\S2.2 Corollary 2]{Broer:DirectSummandProperty}.
	\end{proof}

	A second natural question to ask is whether the normal domain $S^{G}$ is a UFD. 

	\begin{prop} \label{prop:UFD}
		The ring $S^{G}$ is a unique factorisation domain if and only if $\chr(K) = 2$. 
		If $\chr(K)$ is odd, then the class group of $S^{G}$ is isomorphic to $\mathbb{Z}/2$.
	\end{prop}
	\begin{proof} 
		As noted in \Cref{lem:G-is-small}, the action of $G$ contains no pseudoreflections. 
		Thus, by~\Cite[Theorem 3.9.2]{Benson:PolynomialInvariantsBook}, the class group $\Cl(S^{G})$ is isomorphic to $\Hom(\widehat{G}, K^{\times})$. 
		Note that $S^{G}$ is a UFD precisely when $\Cl(S^{G})$ is trivial. 
		Because $K^{\times}$ is abelian, we obtain
		\begin{equation*} 
			\Cl(S^{G}) \cong \Hom(\widehat{G}, K^{\times}) \cong \Hom(\widehat{G}/[\widehat{G}, \widehat{G}], K^{\times}).
		\end{equation*}

		Thus, the result is clear when $K = \mathbb{F}_{2}$ for then $K^{\times}$ is trivial. 
		We now assume that $K$ has at least three elements. 
		Then, by \Cref{lem:commutator}, we have
		\begin{equation*} 
			\frac{\widehat{G}}{[\widehat{G}, \widehat{G}]} 
			\cong \frac{K^{\times}}{\left(K^{\times}\right)^{2}} 
			\cong \frac{\mathbb{Z}/(q-1)}{2\mathbb{Z}/(q - 1)} 
			\cong \mathbb{Z}/(2, q - 1).
		\end{equation*}

		Observing that 
		$\displaystyle\Hom
		\left(
		\frac{\mathbb{Z}}{(2, q - 1)}, 
		\frac{\mathbb{Z}}{(q - 1)}
		\right) 
		\cong \frac{\mathbb{Z}}{\gcd(2, q - 1)}$
		gives us the desired result.
	\end{proof}

	\begin{rem}
		We note that the results on the $a$-invariant and the class group readily generalise to the conjugation action of $\GL_{n}(K)$ on $K[X_{n \times n}]$ for any $n > 2$ and finite field $K$. 
		If we continue to use $G$ and $S$ respectively to denote the group and the polynomial ring, we then obtain:
		\begin{enumerate}[label=(\alph*)]
			\item $a(S^{G}) = a(S) = -n^{2}$ and the inclusion $S^{G} \into S$ does not split, and
			\item $S^{G}$ is a unique factorisation domain precisely when $n$ and $q - 1$ are coprime; more generally, the class group of $S^{G}$ is $\mathbb{Z}/\gcd(n, q - 1)$.
		\end{enumerate}

		The same proofs as in \Cref{prop:a-invariant-invariant,prop:SG-not-F-regular} are easily adapted for (a); 
		the proofs of \Cref{lem:G-is-small,lem:action-factors-through-SL} work with slight modifications to show that the action of $G$ contains no pseudoreflections and factors through $\SL$. 
		For (b), the same proof as \Cref{prop:UFD} works mutatis mutandis noting that we have $\widehat{G} \cong \PGL_{n}(K)$.
	\end{rem}

	\begin{rem}
		Continuing our notation from the previous remark, 
		we now observe that the homological properties of $S^{G}$ do not generalise. 
		If $n > 2$ and $P \le G$ is the unipotent subgroup, 
		then $S^{G}$ is not a complete intersection and 
		$S^{P}$ is not Cohen--Macaulay. 
		These facts follow from~\Cite[Theorem A]{KacWatanabe} and~\Cite[Corollary 3.7]{Kemper:CM}, respectively, after noting that the 
		action of $G$ contains no \emph{bireflections} for $n > 2$; 
		a non-identity element $\sigma \in \GL(V)$ is said to be a bireflection if $\rank(\sigma - \id_{V}) \le 2$.
	\end{rem}

\part{The special linear group and traceless matrices} \label{part:two}

	We now describe how we can obtain other invariant rings from the above. 
	Recall that the Lie algebra of $\SL_{2}(K)$ is the space of \emph{traceless} matrices
	\begin{equation*} 
		\sl \coloneqq \{M \in \MM_{2}(K) : \trace(M) = 0\},
	\end{equation*}
	and is stable under the conjugation action. 
	Thus, we may ask for the polynomial invariants of this conjugation action.
	We set $T \coloneqq \Sym(\sl^{\ast})$,
	and identify $T$ with a polynomial algebra 
	after choosing the basis
	$\smatrix{1 & 0 \\ 0 & -1}$, $\smatrix{0 & 1 \\ 0 & 0}$, $\smatrix{0 & 0 \\ 1 & 0}$ for $\sl$ 
	and letting $\{a, b, c\}$ be the dual basis. 
	We then have $T = K[a, b, c]$, and
	suggestively arranging these variables in a $2 \times 2$ matrix as
	$Y \coloneqq \smatrix{a & b \\ c & -a}$, 
	we get that
	the element $\sigma \in G$ then acts on $T$ by
	\begin{equation*} 
		\sigma \colon
		Y	\mapsto	\sigma^{-1}	Y \sigma.
	\end{equation*}
	Similarly, we may also ask for the polynomial invariants for the conjugation actions of $\SL_{2}(K)$ on $\gl$ and $\sl$. 
	We answer all three questions below. As it turns out, the structures of the invariant rings are very similar, and they all essentially follow from the description and techniques of \Cref{part:one}.

\section{The action of \texorpdfstring{$\GL_{2}$}{GL2} on \texorpdfstring{$\sl$}{sl2}} \label{sec:GL-traceless}
	
	We continue the use of the notation $K$, $G$, $V$, $S$ as in \Cref{part:one}. 
	Because the inclusion $\sl \into V$ is $G$-equivariant, so is the corresponding $K$-algebra homomorphism
	\begin{align*} 
		\pi \colon S & \onto T \\
		\bgeneric	& \mapsto \btraceless,
	\end{align*}
	that is, $\pi(\sigma \cdot f) = \sigma \cdot \pi(f)$ for all $f \in S$ and $\sigma \in G$. 
	In particular, $\pi(S^{G}) \subset T^{G}$. 
	We will show that all invariants arise this way if the characteristic is odd; 
	this will be an easy consequence of the inclusion $\sl \into V$ splitting $G$-equivariantly. 
	In particular, the invariant ring is then a hypersurface.
	If the characteristic is two, then there are more invariants and the invariant ring is even a polynomial ring. 
	The reason for the homological change in characteristic two can be attributed to 
	the action of $\GL_{2}$ on $\sl$ contains pseudoreflections precisely in characteristic two. 
	We continue to denote $\pi(\det)$ by $\det$, i.e., $\det = -a^{2} - b c \in T$.

	\begin{thm} \label{thm:traceless-GL}
		Let $K \coloneqq \mathbb{F}_{q}$ be a finite field with $q$ elements. 
		Consider the conjugation action of $G \coloneqq \GL_{2}(K)$ on
		$T \coloneqq \Sym(\sl^{\ast})$.
		\begin{enumerate}[label=(\roman*)]
			\item If $q$ is even, then $T^{G}$ is a polynomial ring given as
			\begin{equation*} 
				T^{G} = K[
				\det,\;
				\mathcal{P}^{1}(\det),\;
				\sqrt{\pi(f_{4})}].
			\end{equation*}
			The $a$-invariant of $T^{G}$ is $-\left(3 + q + \binom{q}{2}\right)$.
			\item If $q$ is odd, then $\pi(S^{G}) = T^{G}$.
			Thus, $T^{G}$ is a hypersurface given as
			\begin{equation*} 
				T^{G} = K[
				\det,\;
				\mathcal{P}^{1}(\det),\;
				\pi(f_{4}), \;
				\pi(h)],
			\end{equation*}
			where $h$ is any secondary invariant as in \Cref{mainthm}. 
			We may also replace $\pi(h)$ with the Jacobian of the other three invariants above. 
			The Hilbert series of $T^{G}$ is given as
			\begin{equation} \label{eq:hilb-TG}
				\Hilb(T^{G}, z) = 
				\frac{1 + z^{q^{2}}}
				{(1 - z^{2})(1 - z^{q + 1})(1 - z^{q^{2} - q})}.
			\end{equation}
			Additionally, 
			$T^{G}$ is not $F$-regular, 
			has class group $\mathbb{Z}/2$ 
			and $a$-invariant $-3$.
		\end{enumerate}
	\end{thm}

	\begin{proof}[Proof of (ii)] 
		Assume that $q$ is odd. 
		Then, the $2 \times 2$ identity matrix $I_{2}$ is not in $\sl$ and we thus have the decomposition of $K$-vector spaces $V = \sl \oplus K \cdot I_{2}$. 
		Because $K \cdot I_{2}$ is $G$-stable, the injection $\sl \into V$ is split, and thus
		applying the contravariant functor $\left(\Sym\left((-)^{\ast}\right)\right)^{G}$ gives us the claimed surjection $S^{G} \onto T^{G}$. 
		The kernel is seen to be $(a + d) S^{G}$, giving us the Hilbert series of $T^{G}$ from our knowledge of $\Hilb(S^{G})$ by \Cref{mainthm}. 
		This also gives us the $a$-invariant as $T^{G}$ is then a hypersurface, and hence $a(T^{G}) = \deg \Hilb(T^{G})$. 
		In turn, we get the failure of $F$-regularity as in \Cref{prop:SG-not-F-regular} after noting that the action of $G$ continues to be modular. 
		The class group follows as in \Cref{prop:UFD}. 
		Applying these arguments to the induced action of $\Gamma$ on $T$ shows that $T^{\Gamma} = K[\det,	\mathcal{P}^{1}(\det), \pi(f_{4})]$, 
		and thus the Jacobian of these three invariants is nonzero and may be used as the secondary invariant. 
	\end{proof}

	To finish the proof for characteristic two, we now prove the analogous results about the action of $G$ on $\sl$ as we had for $V$. 
	We let $P \le G$ denote the unipotent subgroup as in~\Cref{eq:unipotent-defn}.

	\begin{lem} \label{lem:conjugation-on-traceless-properties}
		Let $q$ be even, and $\varphi \colon G \to \GL(\sl)$ be the conjugation representation of $G$. 
		Then, 
		\begin{enumerate}[label=(\roman*)]
			\item The kernel of $\varphi$ consists precisely of the scalar matrices. 
			In particular, $\md{\varphi(G)} = q(q - 1)(q + 1)$.
			\item The codimension $\dim(\sl) - \dim(\sl^{P})$ is $1$. 
			Consequently, $T^{G}$ is Cohen--Macaulay.
		\end{enumerate}
	\end{lem}
	\begin{proof} 
		The first statement is a straightforward computation. 
		For the second, 
		note that $\sl^{P}$ is the $K$-span of $I_{2}$ and $\smatrix{0 & 1 \\ 0 & 0}$.
		As before, $\codim(\sl^{P}) \le 2$ tells us that $T^{P}$ is Cohen--Macaulay by~\Cite[Theorem 3.9.2]{CampbellWehlau:ModularInvariantTheory}. 
		As in the proof of \Cref{prop:SG-cohen-macaulay}, we are now done.
	\end{proof}

	\begin{lem} 
		If $q$ is even, then $\pi(f_{4})$ is a square, i.e., $\pi(f_{4}) = f^{2}$ for some $f \in T$. 
		Necessarily, $f \in T^{G}$.
	\end{lem}
	\begin{proof} 
		Assume that $\chr(K) = 2$. 
		Then, $\pi(f_{4})$ takes the form
		\begin{equation*} 
			\pi(f_{4}) =
			\prod_{\substack{A \in K \\ B \in K^{\times}}} 
			\left(\tau a + B b + \frac{g(A)}{B} c \right).
		\end{equation*}
		Note that because $\chr(K) = 2$, we must necessarily have $\tau \neq 0$ because every element of $K$ is a perfect square. 
		Thus, the map $A \mapsto \tau - A$ is an involution of $K$ with no fixed points. 
		Because $g(A) = g(\tau - A)$, we see that in the product above, each linear factor has multiplicity exactly two, giving us that $\pi(f_{4})$ is a perfect square. 

		For the last statement, note that $\pi(f_{4})$ is an invariant and square roots are unique in characteristic two.
	\end{proof}

	\begin{lem} 
		When $q$ is even, the invariants $\det$, $\mathcal{P}^{1}(\det)$, $\sqrt{\pi(f_{4})}$ 
		form a system of parameters for $T$.
	\end{lem}
	\begin{proof} 
		This follows from \Cref{thm:primary-invariants} after noting that $T$ is obtained from $S$ by killing the first term of a system of parameters.
	\end{proof}

	The lemmas above assemble as before to finish the proof of \Cref{thm:traceless-GL}.
	
	\begin{proof}[Proof \Cref{thm:traceless-GL} (i)]
		Assume that $q$ is even. 
		In view of \Cref{lem:conjugation-on-traceless-properties}, the product of degrees of our choice of primary invariants is given by
		\begin{equation*} 
			2 \cdot (q + 1) \cdot \frac{(q^{2} - q)}{2} 
			= q(q - 1)(q + 1) = \md{\varphi(G)}.
		\end{equation*}
		Because $T^{G}$ is Cohen--Macaulay, we get that 
		$T^{G} = K[\det,\; \mathcal{P}^{1}(\det),\; \sqrt{\pi(f_{4})}]$, 
		by~\Cite[Theorem 3.7.1]{DerksenKemper}. 
		The $a$-invariant of a polynomial ring is then the negative of the sum of the degrees of the generators.
	\end{proof}

\section{The action of \texorpdfstring{$\SL_{2}$}{SL2} on \texorpdfstring{$\gl$}{gl2} and \texorpdfstring{$\sl$}{sl2}} \label{sec:SL-generic}
	
	We now consider the conjugation action of $\SL_{2}(K)$ on the space $V = \MM_{2}(K)$ of all $2 \times 2$ matrices. 
	As before, we see that the kernel of this action is precisely the scalar matrices, i.e., we obtain a faithful action of $\PSL_{2}(K)$. 
	Because $\PSL_{2}(\mathbb{F}_{q}) = \PGL_{2}(\mathbb{F}_{q})$ when $q$ is even, we get the same invariant ring as in \Cref{mainthm}. 
	For this reason, we now consider $K$ to be a finite field of odd characteristic. 
	In this case, we may fix an irreducible quadratic of the form $g(x) = x^{2} + \delta \in K[x]$. 
	We then define
	\begin{equation} \label{eq:f4-SL}
		f_{4} \coloneqq 
		-\prod_{\substack{A \in K \\ B \in K^{\times}}} 
		\left(A a + B b - \frac{A^{2} + \delta}{B} c - A d\right). \\
	\end{equation}
	Up to a sign, $f_{4}$ above coincides with our choice as in~\Cref{defn:primary-invariants-description}. 
	We obtain the following result.

	\begin{thm} \label{thm:SL-generic-invariants}
		Let $q$ be an odd prime power, 
		$K = \mathbb{F}_{q}$ the finite field with $q$ elements, 
		$G \coloneqq \SL_{2}(K)$ the special linear group, 
		and $S \coloneqq K[X_{2 \times 2}]$ the polynomial ring. 
		For the conjugation action of $G$ on $S$, we have
		\begin{equation*} 
			K[X]^{\SL_{2}(K)} = 
			K
			[\trace,\; 
			\det,\; 
			\mathcal{P}^{1}(\det),\; 
			\sqrt{f_{4}},\; 
			h],
		\end{equation*}
		where $f_{4}$ is as in~\Cref{eq:f4-SL}, 
		and $h$ is the Jacobian of the other four invariants listed. 
		The invariant ring is a hypersurface with Hilbert series
		\begin{equation*} 
			\Hilb(S^{G}, z) = 
			\frac{1 + z^{\binom{q + 1}{2}}}
			{(1 - z)(1 - z^{2})(1 - z^{q + 1})(1 - z^{\binom{q}{2}})}	
		\end{equation*} 
		Additionally, $S^{G}$ is not $F$-regular, is a unique factorisation domain, and has $a$-invariant $-4$.
	\end{thm}

	The proof follows the same recipe as before. 
	The additional ingredient is that $f_{4}$ above is indeed a perfect square, as we now show. 
	We continue to use the same notation as in the theorem. 
	In particular, $K$ is a finite field of odd characteristic.

	\begin{lem} 
		There exists $f \in K[X]^{\SL_{2}(K)}$ such that $f^{2} = f_{4}$.
	\end{lem}
	\begin{proof} 
		The idea is simple: we show that the linear factors appearing in~\Cref{eq:f4-SL} occur in pairs, up to sign;
		thus, we may construct a square root $f \in K[X]$ by considering only half the factors. 
		To this end, for $A \in K$ and $B \in K^{\times}$, we define $\lambda(A, B)$ to be the linear factor given by
		\begin{equation*} 
			\lambda(A, B) \coloneqq A a + B b - \frac{A^{2} + \delta}{B} c - A d.
		\end{equation*}
		Because $\alpha \mapsto - \alpha$ is an involution of $K^{\times}$ with no fixed points, 
		we may decompose $K^{\times}$ as $P \sqcup -P$ 
		for some subset $P \subset K^{\times}$ of cardinality $(q - 1)/2$. 
		Set
		\begin{equation*} 
			f \coloneqq 
			\prod_{\substack{A \in K \\ B \in P}} \lambda(A, B).
		\end{equation*}
		
		We first show that $f^{2} = \pm f_{4}$. 
		Noting that $\lambda(-A, -B) = -\lambda(A, B)$, we get
		\begin{align*} 
			-f_{4} 
			= 
			\prod_{\substack{A \in K \\ B \in P}} \lambda(A, B) 
			\cdot 
			\prod_{\substack{A \in K \\ B \in -P}} \lambda(A, B) 
			&= 
			\prod_{\substack{A \in K \\ B \in P}} \lambda(A, B) 
			\cdot 
			\prod_{\substack{A \in K \\ B \in P}} \lambda(-A, -B).\\
			&= (-1)^{\md{K}\md{P}}
			\prod_{\substack{A \in K \\ B \in P}} \lambda(A, B) 
			\cdot 
			\prod_{\substack{A \in K \\ B \in P}} \lambda(A, B) = (-1)^{\md{P}} f^{2}.
		\end{align*}
		Thus, $f^{2} = (-1)^{\frac{q + 1}{2}} f_{4}$. 
		If $(q + 1)/2$ is even, then we are done. 
		Otherwise, $q \equiv 1 \pmod{4}$ and we have shown $f^{2} = - f_{4}$. 
		However, for such a field, $-1$ has a square root, and we may modify $f$ with a scalar to obtain $f^{2} = f_{4}$. 

		We now show that $f$ is $\SL_{2}(K)$-invariant. 
		Let $\sigma \in \GL_{2}(K)$ be arbitrary. 
		Because $f_{4}$ is $\sigma$-invariant by \Cref{prop:primary-invariants-are-invariant}, we get that $\sigma(f)^{2} = f^{2}$. 
		Thus, $\sigma(f) \in \{f, -f\}$ for all $\sigma \in \GL_{2}(K)$, 
		giving us the homomorphism
		\begin{align*} 
			\chi \colon \GL_{2}(K) &\to K^{\times} \\
			\sigma & \mapsto \frac{\sigma(f)}{f}.
		\end{align*}
		By \Cref{lem:commutator}, we know that the above must be trivial on $\SL_{2}(K)$. 
	\end{proof}

	\begin{proof}[Proof of \Cref{thm:SL-generic-invariants}] 
		The first four invariants listed form a system of primary invariants in view of \Cref{thm:primary-invariants}. 
		We now follow the same steps as in \Cref{sec:homological-properties,sec:missing-invariant}.
		By \Cref{por:conjugation-invariants-CM}, $S^{G}$ is Cohen--Macaulay.
		The product of degrees of the primary invariants is exactly twice the order of $\PSL_{2}(K)$, giving us that $S^{G}$ is a hypersurface with secondary invariants $1$ and $\eta$. 
		We obtain $a(S^{G}) = a(S) = -4$ and conclude $\deg(\eta) = \binom{q + 1}{2}$. 
		The Jacobian of the chosen primary invariants has this degree, and one checks that this works. 
	\end{proof}

	As in \Cref{sec:GL-traceless}, 
	we now consider the conjugation action of the special linear group on the vector space $\sl$ of traceless matrices.
	We continue to denote the corresponding polynomial ring as 
	$T \coloneqq K \traceless$.
	As before, we have the projection $\pi \colon S \onto T$, and we define
	\begin{equation} \label{eq:f4-SL-traceless}
		\widetilde{f}_{4} \coloneqq
		\pi(f_{4}) =  
		-\prod_{\substack{A \in K \\ B \in K^{\times}}} 
		\left(2 A a + B b - \frac{A^{2} + \delta}{B} c \right). \\
	\end{equation}
	In this case, we have the following.

	\begin{thm} \label{thm:SL-traceless-invariants}
		Let $q$ be a power of an odd prime, 
		$K = \mathbb{F}_{q}$ the finite field with $q$ elements, 
		$G \coloneqq \SL_{2}(K)$ the special linear group, 
		and $T \coloneqq \Sym(\sl^{\ast}) = K \traceless$.
		For the conjugation action of $G$ on $T$, we have
		\begin{equation*} 
			T^{G} = 
			K
			[
			\det, 
			\mathcal{P}^{1}(\det), 
			\sqrt{\widetilde{f}_{4}}, 
			h],
		\end{equation*}
		where $\widetilde{f}_{4}$ is as in~\Cref{eq:f4-SL-traceless}, 
		and $h$ is the Jacobian of the other three invariants listed. 
		The invariant ring is a hypersurface with Hilbert series
		\begin{equation*} 
			\Hilb(T^{G}, z) = 
			\frac{1 + z^{\binom{q + 1}{2}}}
			{(1 - z^{2})(1 - z^{q + 1})(1 - z^{\binom{q}{2}})}.
		\end{equation*} 
		Additionally, $T^{G}$ is not $F$-regular, is a unique factorisation domain, and has $a$-invariant $-3$.
	\end{thm}
	\begin{proof} 
		The proof proceeds the same way as that of \Cref{thm:traceless-GL} by making use of the corresponding description from \Cref{thm:SL-generic-invariants}.
	\end{proof}

\part{The actions of \texorpdfstring{$\OO_{2}$}{O2}} \label{part:three}
	
	We now focus on the conjugation action of the \emph{orthogonal group} $\OO_{2}(K)$.
	As usual, we let $K$ denote the finite field with $q$ elements, 
	and define
	\begin{equation*} 
		\OO_{2}(K) \coloneqq 
		\{\sigma \in \GL_{2}(K) : \sigma^{-1} = \sigma^{\tr}\}.
	\end{equation*}

	We remark that there are different notions of the orthogonal group in positive characteristic, see~\Cite[\S2.5]{KleidmanLiebeck}. 
	We shall stick with the above as our definition, though it differs from the standard notions in characteristic two. 
	This is in line with the notion of the orthogonal group in classical invariant theory~\Cite{ConciniProcesiCharacteristicFree}. 

	We use the following bijection to parametrise $G$:
	\begin{equation} \label{eq:O-bijection}
		\begin{aligned}
			\Phi \colon \{(s, t) \in K \times K : s^{2} + t^{2} = 1\} 
			\times 
			\{\pm 1\}
			&\to 
			\OO_{2}(K) \\
			(s, t, \varepsilon) &\mapsto 
			\begin{bmatrix}
				s & t \\
				-\varepsilon t & \varepsilon s
			\end{bmatrix}
		\end{aligned}
	\end{equation}
	The above is a bijection even when $q$ is even, with the understanding that the set $\{\pm 1\}$ is then a singleton.

\section{The action of \texorpdfstring{$\OO_{2}$}{O2} on \texorpdfstring{$\gl$}{gl(2)}} \label{sec:orthogonal-generic}

	We first consider the conjugation action of $G$ 
	on $V \coloneqq \MM_{2}(K)$ and in turn 
	on $S \coloneqq \Sym(V^{\ast}) = K \generic$. 
	As before, we set $\widehat{G} \le \GL(V)$ to be the image of the conjugation representation; the kernel consists precisely of the scalar matrices. 
	The only scalar orthogonal matrices are $\pm I_{2}$, giving us
	\begin{equation*} 
		\mdd{\widehat{G}} = 
		\begin{cases}
			\md{G} & \text{ if $q$ is even}, \\
			\md{G}/2 & \text{ if $q$ is odd}.
		\end{cases}
	\end{equation*}

	With the notation of~\Cref{eq:O-bijection}, we calculate the conjugation action to be
	\begin{equation} \label{eq:O-action}
		\Phi(s, t, \varepsilon) \colon 
		\begin{bmatrix}
			a & b \\
			c & d
		\end{bmatrix}
		\mapsto 
		\begin{bmatrix}
			s^{2} a - s t \varepsilon (b + c) + t^{2} d 
			& st (a - d) + \varepsilon s^{2} b - \varepsilon t^{2} c  \\
			s t (a - d) - \varepsilon t^{2} b + \varepsilon s^{2} c & 
			t^{2} a + \varepsilon s t (b + c) + s^{2} d
		\end{bmatrix}
	\end{equation}
	Thus, $(b - c) \mapsto \varepsilon (s^{2} + t^{2}) (b - c) = \pm (b - c)$. 
	The above also gives us that the stabiliser of $a$ consists of the matrices corresponding to $t = 0$; 
	this gives us the matrices $\smatrix{\pm 1 & 0 \\ 0 & \pm 1}$. 
	Thus, we get the following lemma.

	\begin{lem} \label{lem:orthogonal-invariants-stabiliser}
		We have the following for the $G$-action on $S$.
		\begin{enumerate}[label=(\roman*)]
			\item If $q$ is even, then $b - c$ is invariant, and the orbit of $a$ has size $\md{G}$.
			\item If $q$ is odd, then $(b - c)^{2}$ is invariant, and the orbit of $a$ has size $\md{G}/4$.
			\hfill $\square$
		\end{enumerate}
	\end{lem}

	In view of the above lemma, we now define the following invariants. 
	\begin{equation} \label{eq:O-primary-invariants}
		\begin{aligned} 
			f_{1} &\coloneqq a + d, \\
			f_{2} &\coloneqq ad - b c, \\
			f_{3} &\coloneqq 
			\begin{cases}
				b - c & \text{ if $q$ is even}, \\
				(b - c)^{2} & \text{ if $q$ is odd},
			\end{cases} 
			\\
			f_{4} &\coloneqq N(a),
		\end{aligned}
	\end{equation}
	where $N(a)$ denotes the orbit product of $a$, i.e., the product of all the polynomials in the orbit of $a$. 
	We now show that they form a system of primary invariants.

	\begin{thm} \label{thm:orthogonal-primary-invariants}
		The invariants $f_{1}, \ldots, f_{4}$ form a system of parameters for $S$. 
		We have $\prod_{i = 1}^{4} \deg(f_{i}) = 2 \mdd{\widehat{G}}$.
	\end{thm}
	\begin{proof} 
		We proceed as in the proof of \Cref{thm:primary-invariants} to show that the only common zero of these polynomials is the origin. 
		For $f_{4}$ to vanish, one of its linear factors must vanish. 
		Because the linear factors form a $G$-orbit and the other $f_{i}$ are $G$-invariant, we may assume without loss of generality that $a = 0$. 
		It is now clear that the vanishing of the other polynomials forces $b = c = d = 0$ as well. 
		The second statement follows from the knowledge of $\deg(N(a))$ from \Cref{lem:orthogonal-invariants-stabiliser}. 
	\end{proof}

	\begin{thm} \label{thm:O-generic-invariants}
		Let $K$ be a finite field, $G \coloneqq \OO_{2}(K)$ the orthogonal group acting via conjugation on $S \coloneqq K \generic$. 
		The invariant ring $S^{G}$ is a hypersurface of the form $K[f_{1}, \ldots, f_{4}, h]$, 
		where the $f_{i}$ are as in~\Cref{eq:O-primary-invariants}, 
		and $h$ is an invariant of degree 
		$-4 + \sum_{i = 1}^{4} \deg(f_{i})$. 
		The Hilbert series of $S^{G}$ is given as
		\begin{equation*} 
			\Hilb(S^{G}, z) 
			= 
			\frac{1 + z^{\deg h}}
			{(1 - z)(1 - z^{2})(1 - z^{\deg f_{3}})(1 - z^{\deg f_{4}})}
		\end{equation*}
		Additionally, the invariant ring $S^{G}$ 
	  has $a$-invariant $-4$,
		is $F$-regular precisely when $\chr(K) \neq 2$,
	  and is a unique factorisation domain precisely when $\chr(K) = 2$. 
	  When $\chr(K) \neq 2$, we may take $h$ to be $\Jac(f_{1}, \ldots, f_{4})$, 
	  and the class group of $S^{G}$ is $(\mathbb{Z}/2)^{2}$. 
	\end{thm}
	\begin{proof} 
		By \Cref{por:conjugation-invariants-CM,por:conjugation-invariants-a-invariant}, 
		we know that $S^{G}$ is Cohen--Macaulay with $a$-invariant~$-4$. 
		The rest of the proof is as we have seen in \Cref{part:one}:
		By \Cref{thm:orthogonal-primary-invariants}, we see that $S^{G}$ is a hypersurface with the degree of $h$ as listed. 
		Because $a(S^{G}) = a(S)$,
		we see that $S^{G}$ is $F$-regular precisely when the action is nonmodular. 
		In view of \Cref{lem:order-orthogonal-group}, this is equivalent to $\chr(K) \neq 2$. 
		Next, we have $\Cl(S^{G}) \cong \Hom(\widehat{G}/[\widehat{G}, \widehat{G}], K^{\times})$; 
		this is $(0)$ when $q$ is even because $\widehat{G}$ and $K^{\times}$ then have coprime orders. 
		When $q$ is odd, it is known that $G/[G, G] \cong (\mathbb{Z}/2)^{2}$, by \Cite[pp. 50, 55]{Dieudonne:GeometrieClassiques}.
		Because $-I$ is the commutator of $\smatrix{0 & 1 \\ -1 & 0}$ and $\smatrix{1 & 0 \\ 0 & -1}$, 
		we get that 
		$\widehat{G}/[\widehat{G}, \widehat{G}] \cong 
		G/[G, G] \cong 
		(\mathbb{Z}/2)^{2}$ 
		and the result follows after noting that $K^{\times}$ is a cyclic group of even order. 
	\end{proof}

	We now record the order of $G$ and the degrees of the above invariants. 

	\begin{lem} \label{lem:order-orthogonal-group}
		Continuing with the earlier notation, we have the following.
		\begin{enumerate}[label=(\roman*)]
			\item If $q$ is even, then $\md{G} = \mdd{\widehat{G}} = q$. 
			We have $\deg(f_{3}) = 1$, 
			$\deg(f_{4}) = q$, 
			and $\deg(h) = q$.
			\item If $q \equiv 1 \pmod{4}$, 
			then $\md{G}/2 = \mdd{\widehat{G}} = q - 1$. \newline
			We have $\deg(f_{3}) = 2$, 
			$\deg(f_{4}) = (q - 1)/2$, 
			and $\deg(h) = (q + 1)/2$.
			\item If $q \equiv -1 \pmod{4}$, 
			then $\md{G}/2 = \mdd{\widehat{G}} = q + 1$. \newline
			We have $\deg(f_{3}) = 2$, 
			$\deg(f_{4}) = (q + 1)/2$, 
			and $\deg(h) = (q + 3)/2$.
		\end{enumerate}
	\end{lem}
	\begin{proof}
		We only need to calculate $\md{G}$.
		In view of~\Cref{eq:O-bijection}, this comes down to calculating the size of the circle 
		${S \coloneqq \{(s, t) \in K \times K : s^{2} + t^{2} = 1\}}$. 
		When $q$ is even, the condition $s^{2} + t^{2} = 1$ is equivalent to $s + t = 1$ and the result follows. 
		When $q$ is odd, it is well-known that the cardinality of $S$ is $q - (-1)^{(q-1)/2}$; 
		alternatively, the order of $G$ may be found in~\Cite[\S6.10]{Jacobson:BAI} after noting that $-1$ is a square in $K$ precisely when $q \equiv 1 \pmod{4}$.
	\end{proof}

\section{The action of \texorpdfstring{$\OO_{2}$}{O2} on \texorpdfstring{$\oo$}{o(2)} and symmetric matrices}
	
	The polynomial invariants for the adjoint representation of $G = \OO_{2}(K)$ are easily calculated.

	\begin{rem} \label{rem:O-adjoint-invariants}
		The group $G$ is one-dimensional and its Lie algebra $\oo \le \MM_{2}(K)$ consists of alternating matrices; 
		in fact, the adjoint representation is simply the map $\det \colon G \to K^{\times}$. 
		Thus, if we let $K[b]$ denote $\Sym(\oo^{\ast})$, then $K[b]^{G}$ is either $K[b^{2}]$ or $K[b]$ according to whether the characteristic is either odd or even. 
	\end{rem}

	A more interesting action is obtained by noticing that the space of symmetric matrices is stable under conjugation by the orthogonal group. 
	We continue with the notation from the previous section and further define
	\begin{equation*} 
		\sym \coloneqq \{M \in V : M = M^{\tr}\} \andd T \coloneqq \Sym(\sym^{\ast}).
	\end{equation*}

	The invariant ring turns out to be a polynomial ring, regardless of the characteristic, as we now show. 
	We have the suggestive coordinates for $T$ given as 
	$T = K \symmetric$. 
	As in \Cref{sec:GL-traceless}, we have the $G$-equivariant $K$-algebra map $\pi \colon S \to T$ given by
	$
		\smatrix{
			a & b \\
			c & d
		}
		\mapsto 
		\smatrix{
			a & b \\
			b & d
		}
	$. 
	Moreover, as before, the inclusion $W \into V$ splits $G$-equivariantly in odd characteristic, and we obtain the invariant ring as $T^{G} = \pi(S^{G})$ when $\chr(K)$ is odd. 
	The following theorem describes the invariant ring. 
	As before, we continue to interpret $\trace$, $\det$, $N(a)$ as elements of $T$. 

	\begin{thm} \label{thm:O-symmetric-invariants}
		Let $K \coloneqq \mathbb{F}_{q}$ be a finite field with $q$ elements. 
		Consider the conjugation action of $G \coloneqq \OO_{2}(K)$ on $T = K[X_{2 \times 2}^{\text{sym}}] = K \symmetric$. 
		Then, $T^{G}$ is a polynomial ring given as
		\begin{equation*} 
			T^{G} =
			K[
			\trace,\; 
			\det,\; 
			f_{3}
			],
		\end{equation*}
		where $f_{3}$ is chosen as follows:
		\begin{enumerate}[label=(\roman*)]
			\item If $q$ is odd, then $f_{3} = N(a)$; 
			we have $\deg(f_{3}) = \md{G}/4$, and
			$a(S^{G}) = -\left(3 + \frac{1}{2}\left(q - (-1)^{(q-1)/2}\right)\right)$.
			\item If $q = 2^{e + 1}$ with $e \ge 0$, then
			\begin{equation*} 
				f_{3} 
				\coloneqq 
				\sum_{k = 0}^{e} b^{2^{k}} (a + d)^{2^{e} - 2^{k}};
			\end{equation*}
			we have $\deg(f_{3}) = \md{G}/2 = q/2$, and
			$a(S^{G}) = -\left(3 + \frac{q}{2}\right)$.
		\end{enumerate}
	\end{thm}
	\begin{proof} 
		Note that $\smatrix{1 & 0 \\ 0 & 1} \in W^{G}$, and so $\codim(W^{G}) \le 2$, giving us that $T^{G}$ is Cohen--Macaulay. 
		In either case, the product of the degrees of the claimed generators is equal to $\md{G/\{\pm I\}}$, 
		and the $a$-invariant is the negative of the sum of degrees. 
		Thus, it suffices to show that the claimed generators are indeed invariant and that they form a system of parameters. 
		For odd characteristic, this follows from the work done in \Cref{sec:orthogonal-generic}.
		Thus, we now assume that $q = 2^{e + 1}$. 
		In view of~\Cref{eq:O-bijection} and~\Cref{eq:O-action}, 
		the group $G$ consists of matrices of the form 
		$\varphi(t) = \smatrix{1 - t & t \\ t & 1 - t}$ for $t \in K$, 
		and we have $\varphi(t) \cdot b = b + (t - t^{2})(a + d)$. 
		As $a + d$ is invariant, the Freshman's Dream~\Cite{FreshmanDream} yields
		\begin{align*} 
			\varphi(t) \cdot f_{3} = 
			\sum_{k = 0}^{e} 
			\left(b^{2^{k}} + (t - t^{2})^{2^{k}} (a + d)^{2^{k}}\right) 
			(a + d)^{2^{e} - 2^{k}} 
			= f_{3} + \sum_{k = 0}^{e} (t^{2^{k}} - t^{2^{k + 1}})(a + d)^{2^{e}}.
		\end{align*}
		The sum telescopes to give us $\varphi(t) \cdot f_{3} = f_{3} + (t - t^{q}) (a + d)^{q/2}$. 
		Because $K$ is the field with $q$ elements, we have $t = t^{q}$, giving us the invariance of $f_{3}$. 
		It is clear that they form a system of parameters because 
		$(a + d, ad - b^{2}, f_{3}) T = (a + d, ad - b^{2}, b^{2^{e}}) T$, 
		giving us the radical as $(a, b, d) T$.
	\end{proof}

\appendix

\section{Steenrod operations} \label{sec:steenrod}

		In this section, we describe how one may produce new invariants from old using the Steenrod operations, a feature available over finite fields. 
		A reference for the material presented here is \Cite[Chapter 11]{Smith:PolynomialInvariantsBook}. 

		Let $q$ be a prime power, and $K = \mathbb{F}_{q}$ the finite field with $q$ elements. 
		Let $S = K[x_{1}, \ldots, x_{n}]$ be a polynomial ring over $K$, and $S[T]$ the polynomial ring obtained by adjoining an additional variable. 
		We define the $K$-algebra map $\mathcal{P} \colon S \to S[T]$ by defining it on the variables as $x_{i} \mapsto x_{i} + x_{i}^{q}T$. 
		This definition is coordinate-free in the sense that $\mathcal{P}(x) = x + x^{q} T$ for any homogeneous linear element $x \in S$. 
		For any nonnegative integer $i \ge 0$ and $f \in S$, we define $\mathcal{P}^{i}(f)$ to be the coefficient of $T^{i}$ in $\mathcal{P}(f)$. 
		In other words, we have $K$-linear maps $\mathcal{P}_{i} \colon S \to S$ satisfying,
		for all $f \in S$,
		the equation
		\begin{equation*} 
			\mathcal{P}(f) = \sum_{i \ge 0} \mathcal{P}^{i}(f) T^{i}.
		\end{equation*}

		These operations are natural in the following sense: 
		Let $R = K[y_{1}, \ldots, y_{m}]$ be a polynomial ring, and $\varphi \colon S \to R$ a degree-preserving $K$-algebra map. 
		Then, the diagram below on the left commutes.
		In turn, so does the diagram below on the right, for all $i \ge 0$.
		\begin{equation*} 
			\begin{tikzcd}
				S \arrow[r, "\varphi"] \arrow[d, "\mathcal{P}"'] & R \arrow[d, "\mathcal{P}"] \\
				S[T] \arrow[r, "{\varphi[T]}"'] & R[T]
			\end{tikzcd}
			\qquad \qquad
			\begin{tikzcd}
				S \arrow[r, "\varphi"] \arrow[d, "\mathcal{P}^{i}"'] & R \arrow[d, "\mathcal{P}^{i}"] \\
				S \arrow[r, "{\varphi}"'] & R
			\end{tikzcd}
		\end{equation*}
		In particular, if $\varphi \colon S \to S$ is a degree-preserving $K$-algebra automorphism and $f \in S$ is fixed by $\varphi$, then so is $\mathcal{P}^{i}(f)$ for any $i$. 
		This lets us produce new invariants from old, giving us the following.

		\begin{lem} \label{lem:steenrod-invariant}
			Let $K$ be a finite field, $S$ a polynomial ring over $K$, and $G$ a group acting on $S$ by degree-preserving $K$-algebra automorphisms. 
			If $f \in S^{G}$ is an invariant, then $\mathcal{P}^{i}(f) \in S^{G}$ for all $i \ge 0$. 
			\qed
		\end{lem}

\printbibliography
\end{document}